\documentclass[12pt]{amsart}
\usepackage[colorlinks=true,citecolor=green,linkcolor=magenta]{hyperref}
\usepackage{amsmath}
\usepackage{verbatim}
\usepackage{amsfonts}
\usepackage{amssymb}
\usepackage{blkarray}
\usepackage{color,colortbl}
\usepackage[all]{xy}  
\usepackage{enumerate}
\usepackage[top=1in, bottom=1in, left=1in, right=1in]{geometry}
\usepackage{mathrsfs}

\usepackage{stmaryrd}
\usepackage{bbold}

\usepackage{cleveref}
\usepackage{soul}

\theoremstyle{definition}
\newtheorem{theorem}{Theorem}[section]
\newtheorem{theoremx}{Theorem}

\numberwithin{equation}{section}

\newtheorem{question}[theorem]{Question}
\newtheorem{corollary}[theorem]{Corollary}
\newtheorem{lemma}[theorem]{Lemma}
\newtheorem{proposition}[theorem]{Proposition}

\theoremstyle{definition}

\newtheorem{example}[theorem]{Example}
\newtheorem{discussion}[theorem]{Discussion}

\newtheorem{conjecture}[theorem]{Conjecture}
\newtheorem{remark}[theorem]{Remark}

\newtheoremstyle{TheoremNum}
        {8pt}{8pt}              
        {\upshape}                      
        {}                              
        {\bfseries}                     
        {.}                             
        {.5em}                             
        {\theoremname{#1}\theoremnote{ \bfseries #3}}
  \theoremstyle{TheoremNum}

\newcommand{\m}{\mathfrak{m}}

\newcommand{\Ext}{\operatorname{Ext}}

\newcommand{\Min}{\operatorname{Min}}

\newcommand{\ra}{\mathcal{R}}
\DeclareMathOperator{\sym}{Sym}

\newcommand{\blue}[1]{\textcolor{cyan}{#1}}

\DeclareMathOperator{\ch}{char}
\DeclareMathOperator{\im}{im}

\newcommand{\myvector}[2]{\begin{blockarray}{c} \blue{#1} \\ \begin{block}{[c]} #2 \\ \end{block}\end{blockarray}}

\title[Harbourne's Conjecture and the containment problem]{A stable version of Harbourne's Conjecture and the containment problem for space monomial curves}

\author[Grifo]{Elo\'isa Grifo}
\address{Department of Mathematics, University of California, Riverside, CA 92521, USA}
\email{eloisa.grifo@ucr.edu}


\begin{document}

\begin{abstract}
	The symbolic powers $I^{(n)}$ of a radical ideal $I$ in a polynomial ring consist of the functions that vanish up to order $n$ in the variety defined by $I$. These do not necessarily coincide with the ordinary algebraic powers $I^n$, but it is natural to compare the two notions. 
	The containment problem consists of determining the values of $n$ and $m$ for which $I^{(n)} \subseteq I^m$ holds. When $I$ is an ideal of height $2$ in a regular ring, $I^{(3)} \subseteq I^2$ may fail, but we show that this containment does hold for the defining ideal of the space monomial curve $(t^a, t^b, t^c)$. More generally, given a radical ideal $I$ of big height $h$, while the containment $I^{(hn-h+1)} \subseteq I^n$ conjectured by Harbourne does not necessarily hold for all $n$, we give sufficient conditions to guarantee $I^{(hn-h+1)} \subseteq I^n$ for $n \gg 0$.
\end{abstract}

\maketitle

\section{Introduction}\label{section intro}

Given a radical ideal $I$ in a domain $R$, the {\bf $n$-th symbolic power} of $I$ is the ideal given~by
$$I^{(n)} \quad = \bigcap_{P \in \Min(I)} \left( I^n R_P \cap R \right).$$
This is the intersection of the minimal components of the ordinary power $I^n$, where minimal stands for non-embedded rather than height minimal. There are many reasons to consider the symbolic powers of an ideal. If $R$ is a polynomial ring, $I^{(n)}$ is the set of functions that vanish up to order $n$ on the variety defined by $I$, by the Zariski--Nagata Theorem \cite{Zariski,Nagata}. For a survey on symbolic powers, see \cite{SurveySP}.

In general, $I^n \neq I^{(n)}$, although $I^n \subseteq I^{(n)}$ always holds. The containment problem deals with statements of the form $I^{(a)} \subseteq I^b$: given an ideal $I$ and a value $b$, one would like to determine the smallest $a$ for which $I^{(a)} \subseteq I^b$ holds. This question turns out to be surprisingly difficult to answer even over a regular ring. There is, however, an upper bound, depending on {\bf big height} of $I$, which is the largest height of an associated prime of $I$.

\begin{theorem}[Ein-Lazarsfeld-Smith, Hochster-Huneke, Ma-Schwede \cite{ELS,comparison,MaSchwede}]\label{ELSHHMS}
	Let $R$ be a regular ring and $I$ a radical ideal in $R$ of big height $h$. Then for all $n \geqslant 1$,
	$$I^{(hn)} \subseteq I^n.$$
\end{theorem}

This result, however, does not give a complete answer to the containment problem. The first interesting case is in dimension $3$, where any prime ideal of height $2$ that is not a complete intersection satisfies $P^{(n)} \neq P^n$ for all $n \geqslant 2$ \cite{Huneke1986}, while Theorem \ref{ELSHHMS} shows that $P^{(2n)} \subseteq P^n$, and in particular that $P^{(4)} \subseteq P^2$. Examples suggest this could be tightened.

\begin{question}[Huneke, 2000]\label{Craig's question}
	If $P$ is a prime of height $2$ in a regular local ring, is $P^{(3)} \subseteq P^2$?
\end{question}

Brian Harbourne \cite[Conjecture 8.4.3]{Seshadri} extended Question \ref{Craig's question} to a more general statement for arbitrary homogeneous ideals in $k[\mathbb{P}^N]$. In the years since, versions of Harbourne's question with the same bound and similar hypotheses have been considered by experts, often restricting to radical ideals:

\begin{conjecture}[Harbourne]\label{Harbourne's Conjecture}
	Let $I$ be a radical ideal of big height $h$ in a regular ring $R$. Then for all $n \geqslant 1$, 
	$$I^{(hn-h+1)} \subseteq I^n.$$
\end{conjecture}

The value suggested by this conjecture is very natural. In fact, Hochster and Huneke's proof \cite{comparison} of Theorem \ref{ELSHHMS} uses the fact that in prime characteristic $p$, $I^{(hq)} \subseteq I^{[q]}$ for all $q = p^e$; this boils down to a beautiful application of the Pigeonhole Principle, which naturally yields the sharper containment $I^{(hq-h+1)} \subseteq I^{[q]}$. However, Conjecture \ref{Harbourne's Conjecture} can fail; Dumnicki, Szemberg, and Tutaj-Gasi\'nska \cite{counterexamples} found the first counterexample to $I^{(3)} \subseteq I^2$, and others followed, such as \cite{HaSeFermat,MalaraSzpond,BenCounterexample,Akesseh,MalaraSzpond,RealsCounterexample}. Several of these counterexamples turn out to belong to the class of radical ideals of height $2$ that arise as the singular loci of reflection arrangements of a finite complex reflection group; Drabkin and Seceleanu have recently given a complete characterization of which ideals in this class satisfy $I^{(3)} \subseteq I^2$ \cite{DrabkinSeceleanu}.

Despite all this, there are no known counterexamples to Question \ref{Craig's question}, which remains open even in dimension $3$.

\begin{theoremx}[see Theorem \ref{thm space monomial curves}]
	Let $k$ be a field of characteristic not $3$, and consider $R = k \llbracket x, y, z \rrbracket$ or $R = k [x, y, z]$.
	Let $P$ be the prime ideal in $R$ defining the space monomial curve $x=t^a$, $y = t^b$ and $z = t^c$. Then 
	$$P^{(3)} \subseteq P^2.$$
\end{theoremx}

This result will follow once we establish sufficient conditions for $I^{(n)} \subseteq I^m$ to hold for each $n > m$ whenever $I$ is generated by the maximal minors of a $2 \times 3$ matrix. This is done in Section \ref{section dim 3 matrix conditions}, following Alexandra Seceleanu's methods \cite{Seceleanu} closely.

More generally, Conjecture \ref{Harbourne's Conjecture} does hold whenever the ideal $I$ has nice properties: if $I$ is the defining ideal of a general set of points in $\mathbb{P}^2$ \cite{BoH} or $\mathbb{P}^3$ \cite{Dumnicki2015}, or if $R$ is of prime characteristic $p$ (respectively, essentially of finite type over a field of characteristic $0$) and $R/I$ is F-pure (respectively, of dense F-pure type) \cite{GrifoHuneke}. Moreover, there are no counterexamples to $I^{(hn-h+1)} \subseteq I^n$ for $n \gg 0$. One might then ask if requiring that $I^{(hk-h+1)} \subseteq I^k$ holds for \emph{some} $k$ is enough to guarantee that $I^{(hn-h+1)} \subseteq I^n$ for $n \gg 0$.

\begin{theoremx}[see Theorem \ref{thm stable}]\label{thmB}
	Let $R$ be a regular ring containing a field, and let $I$ be a radical ideal in $R$ with big height $h$. If
	$$I^{(hm-h)} \subseteq I^{m}$$
	for some $m \geqslant 2$, then
	$$I^{(hk-h)} \subseteq I^{k}$$
	for all $k \gg 0$ (indeed, for all $k \geqslant hm$).
\end{theoremx}

We also show this statement is not vacuous. In particular, in Section \ref{section space monomial curves} we find classes of space monomial curves $(t^a, t^b, t^c)$ satisfying such conditions. Using the methods from \cite{GrifoHuneke}, we give a refinement of Theorem \ref{thmB} whenever $R$ is of prime characteristic and $R/I$ is F-pure. We discuss this and other evidence pointing towards a stable version of Conjecture \ref{Harbourne's Conjecture} in Section \ref{section stable}.

\section{A stable version of Harbourne's Conjecture}\label{section stable}

In this section, we will study the following stable version of Harbourne's Conjecture:

\begin{conjecture}\label{stable conjecture}
	Let $I$ be a radical ideal of big height $h$ in a regular ring $R$. Then
	$$I^{(hn-h+1)} \subseteq I^n$$
	for all $n$ sufficiently large.
\end{conjecture}

\vspace{1em}

Conjecture \ref{stable conjecture} holds 
\begin{enumerate}
	\item if $I^{(hm-h)} \subseteq I^m$ holds for some value $m$ (see Theorem \ref{thm stable}),
	\item if $I^{(hm-h+1)} \subseteq I^m$ for some $m$ and $I^{(n+h)} \subseteq I I^{(n)}$ for all $n \geqslant m$ (see Discussion \ref{remark n+h}), or
	\item if the resurgence $\rho(I)$ satisfies $\rho(I)<h$ (see Remark \ref{remark any C}).
\end{enumerate}

\vspace{1em}

More generally, we will study the following question:

\begin{question}\label{question any C}
	Let $I$ be a radical ideal of big height $h$ in a regular ring $R$. Given an integer $C > 0$, does there exist $N$ such that
	$$I^{(hn-C)} \subseteq I^n$$
	for all $n \geqslant N$?
\end{question}

\vspace{1em} 

The answer to Question \ref{question any C} is affirmative
\begin{enumerate}
	\item if the resurgence $\rho(I)$ satisfies $\rho(I)<h$ (see Remark \ref{remark any C}),
	\item if $I^{(hm-C)} \subseteq I^m$ for some $m$ and $I^{(n+h)} \subseteq I I^{(n)}$ for all $n \geqslant 1$ (see Discussion \ref{remark n+h}), and in particular
	\item if $I^{(hm-C)} \subseteq I^m$ for some $m$, $R$ has characteristic $p > 0$, and $R/I$ is an F-pure ring (see Theorem \ref{thm Fpure stable}).
\end{enumerate}

\vspace{1em}

Moreover, we are not aware of any examples of ideals $I$ for which the answer to Question \ref{question any C} is negative.

\begin{example}
	Harbourne and Seceleanu \cite{HaSeFermat} found ideals $I$ with $I^{(hn-h+1)} \nsubseteq I^n$ for $n$ arbitrarily large; however, their ideals $I$ depend on the choice of $n$. Moreover, as shown in \cite[Theorem 3.2]{DHNSST2015}, Harbourne and Seceleanu's examples still satisfy $\rho(I)<h$, which as we will see guarantee imply that Conjecture \ref{stable conjecture} holds.
\end{example}

\begin{remark}
	Whenever we assume that $I$ is a radical ideal of big height $h$, one may instead take $I$ to be any ideal, in which case $h$ should be replaced by the maximum of all the analytic spreads of $I_P$, where $P$ varies over the set of associated primes of $I$. This is what Hochster and Huneke call the key number \cite[Discussion 1.1]{HHfine} of $I$. However, we write our results for radical ideals of big height $h$ with the goal of improving readability.
\end{remark}

We start by proving that Conjecture \ref{stable conjecture} holds if $I^{(hm-h)} \subseteq I^m$ for some $m$. In fact, the following stronger statement holds:

\newpage

\begin{theorem}\label{thm stable}
	Let $R$ be a regular ring containing a field, and let $I$ be a radical ideal in $R$ with big height $h$. If
	$$I^{(hm-h)} \subseteq I^{m}$$
	for some $m \geqslant 2$, then
	$$I^{(hk-h)} \subseteq I^{k}$$
	for all $k \geqslant hm$.
\end{theorem}
	
\begin{proof}
The key ingredient we will need is the following generalization \cite{Johnson} of Theorem \ref{ELSHHMS}: given any $n \geqslant 1$ and $a_1, \ldots, a_n \geqslant 0$, we have
$$I^{(hn + a_1 + \cdots + a_n)} \subseteq I^{(a_1 + 1)} \cdots I^{(a_n+1)}.$$
	Fix $k \geqslant hm$, and write $k=hm+t$ for some $t \geqslant 0$. Apply the formula above for $n=h+t$, $a_1= \cdots = a_h = hm-h-1$ and $a_{h+1} = \cdots = a_{h+t} = 0$. With these values,
	$$hn + a_1 + \cdots + a_n = h(h+t)+h(hm-h-1) = h(hm+t-1) = hk-h,$$
	and the formula becomes
	$$I^{(hk-h)} \subseteq \left( I^{(mh-h)} \right)^h I^{t}.$$
	By assumption, $I^{(mh-h)} \subseteq I^{m}$. Then
	$$I^{(hk-h)} \subseteq \left( I^{(mh-h)} \right)^h I^{t} \subseteq \left( I^{m} \right)^h I^{t} = I^{mh+t} = I^{k}.$$
\end{proof}

In particular, given an ideal $I$, the containment $I^{(hn-h+1)} \subseteq I^n$ holds for $n \gg 0$ as long as we can find one value $m$ such that $I^{(hm-h)} \subseteq I^m$. 

\begin{example}
	Let $k$ be a field. The monomial ideal
	$$I = \bigcap_{i \neq j} \left( x_i, x_j \right) = \left( x_1 \cdots \widehat{x_i} \cdots x_v \,\, \big| \,\, 1 \leqslant i \leqslant v \right)$$
	in $v$ variables has big height $2$, and since monomial ideals satisfy Conjecture \ref{Harbourne's Conjecture}, by \cite[Example 8.4.5]{Seshadri}, $I^{(2n-1)} \subseteq I^n$ for all $n \geqslant 1$. However, as noted in \cite[Example 1.2]{HHfine} or \cite[Example 3.5]{GrifoHuneke}, $I^{(2n-2)} \nsubseteq I^n$ for all $n<v$. We claim that $I^{(2v-2)} \subseteq I^v$.
	
	First, note that $I^{(2v-2)}$ is generated by the monomials $x_1^{a_1} \cdots x_v^{a_v}$ with $a_i + a_j = 2v-2$ for all $i \neq j$. We will show that all such monomials are in $I^{v}$, which will prove our claim. One of those monomials is $\left( x_1 \cdots x_v \right)^{v-1} \in I^{(2v-2)}$, which we can rewrite as
	$$\left( x_1 \cdots x_v \right)^{v-1} = \prod_{i=1}^v \left( x_1 \cdots \widehat{x_i} \cdots x_v \right) \in I^v.$$
	Given any other generator $x_1^{a_1} \cdots x_v^{a_v}$ of $I^{(2v-2)}$, $a_i < v-1$ for some $i$; fix such $i$. For all $j \neq i$, $x_1^{a_1} \cdots x_v^{a_v} \in \left( x_i, x_j \right)^{2v-2}$, and thus $a_j \geqslant 2v-2-a_i > v-1$. Therefore, the given monomial is a multiple of $\left( x_1 \cdots \widehat{x_i} \cdots x_v \right)^v \in I^v$.
	
	We can now apply Theorem \ref{thm stable} to conclude that $I^{(2n-2)} \subseteq I^n$ for all $n \gg 0$. In particular, $I^{(2n-2)} \subseteq I^n$ for all $n \geqslant 2v$.
\end{example}

\vspace{1em}

We will see other classes of ideals satisfying the conditions of Theorem \ref{thm stable} in Section \ref{section dim 3 matrix conditions}.

\vspace{1em}

We now turn our attention to Question \ref{question any C}: given $C$, does $I^{(hn-C)} \subseteq I^n$ hold for $n \gg 0$? The answer is yes whenever an invariant of $I$ known as the resurgence does not take its largest possible value.

\begin{remark}\label{remark any C}
	The \emph{resurgence} of $I$, first defined by Bocci and Harbourne in \cite{BoH}, is
	$$\rho(I) := \sup \left\lbrace \frac{a}{b} \,\, \big| \,\, I^{(a)} \nsubseteq I^b \right\rbrace.$$
	Over a regular ring, if $I$ has big height $h$ then $1 \leqslant \rho(I) \leqslant h$, where the last inequality is a consequence of Theorem \ref{ELSHHMS}. But as long as $\rho(I)< h$, then in fact $I^{(hn-C)} \subseteq I^n$ holds for
	$$n > \displaystyle\frac{C}{h-\rho(I)}.$$
	Indeed, for such $n$ we have
	$$\frac{hn-C}{n} > \rho(I),$$
	and by definition of resurgence, this implies that $I^{(hn-C)} \subseteq I^n$.
Therefore, proving $\rho(I)<h$ is enough to answer Question \ref{question any C} affirmatively, which is useful since $\rho(I)$ can sometimes be computed or bounded without explicitly solving the containment problem.
		
	The same conclusion holds if instead we consider the potentially smaller quantity
	$$\rho''(I) := \limsup_{b \rightarrow \infty} \rho_b(I) \quad \textrm{ where } \quad \rho_b(I) := \sup \left\lbrace \frac{a}{b} \,\, \big| \,\, I^{(a)} \nsubseteq I^b, a \geqslant 1 \right\rbrace.$$
	
	 Indeed, if this value $\rho''(I)$ is strictly less than $h$, then there exists some $N$ such that
	$$\sup \left\lbrace \frac{a}{b} \,\, \big| \,\, I^{(a)} \nsubseteq I^b, b \geqslant N\right\rbrace < h,$$
	which is enough to conclude that given any $C$, $I^{(hn-C)} \subseteq I^n$ holds for all $n$ large enough. Rephrasing, Question \ref{question any C} has a positive answer as long as 
	$$\limsup_b \, \frac{\max \left\lbrace a \geqslant 1 \,\, \big| \,\, I^{(a)} \nsubseteq I^b \right\rbrace}{b} < h.$$

	As pointed out by the anonymous referee, this quantity $\rho''(I)$ agrees with another invariant, first defined in \cite{AsymptoticResurgence}: 
	$${\rho}_a'(I) := \limsup_t \rho(I,t), \textrm{ where } \rho(I,t) := \sup \left\lbrace \frac{a}{b} \,\, \big| \,\, I^{(a)} \nsubseteq I^b \textrm{ and } a, b \geqslant t \right\rbrace.$$
	To see that $\rho''(I) = \rho_a'(I)$, first note that $I^{(m-1)} \nsubseteq I^m$ for all $m$ implies both $\rho''(I) \geqslant 1$ and $\rho_a'(I) \geqslant 1$. Moreover, $\rho(I,t)$ is a non-increasing sequence on $t$, and thus $\rho_a'(I) = \displaystyle\lim_{t \rightarrow \infty} \rho(I,t)$.

	Suppose $\rho_a'(I) > \rho''(I)$. Since $\rho_a'(I)$ is the limit of the nonincreasing sequence $\rho(I,t)$, then $\rho(I,t) > \rho''(I)$ for all $t$. In particular, there exist $a_1, b_1 \geqslant 1$ such that $a_1/b_1 > \rho''(I)$ and $I^{(a_1)} \nsubseteq I^{b_1}$. We inductively construct sequences $(a_n)$, $(b_n)$ as follows: since $\rho(I,b_{i}+1) > \rho''(I)$ we can pick $b_{i+1} > b_i$ such that $a_{i+1}/b_{i+1} > \rho''(I)$ and $I^{(a_{i+1})} \nsubseteq I^{b_{i+1}}$. These conditions also imply $a_{i+1}/b_{i+1} \leqslant \rho_{b_i}(I)$, and since $\limsup \rho(I,b) > \rho''(I)$, we can chose $(a_i)$ and $(b_i)$ in such a way that guarantees $\limsup_{i \rightarrow \infty} \frac{a_i}{b_i} > \rho''(I)$. This leads to a contradiction:
	$$\rho''(I) = \limsup_{b \rightarrow \infty} \rho_b(I) \geqslant \limsup_{i \rightarrow \infty} \rho_{b_i}(I) \geqslant \limsup_{i \rightarrow \infty} \,\frac{a_i}{b_i} > \rho''(I).$$
	
	This shows that $\rho_a'(I) \leqslant \rho''(I)$. If $\rho''(I) =1$, we are done. If $\rho''(I)>1$, there exist an increasing sequence $(b_n)$ and $a_n > b_n$ such that $a_n / b_n \rightarrow \rho''(I)$, and thus by definition we must have $a_n/b_n \leqslant \rho(I,b_n)$ and $\rho''(I) \leqslant \rho_a'(I)$.

	Summarizing, we have seen that if we show that either $\rho(I)<h$ or $\rho''(I)=\rho_a'(I)<h$, then in fact for any positive integer $C$, the containment $I^{(hn-C)} \subseteq I^n$ holds for all $n$ sufficiently large.
	One could also try to bound the asymptotic resurgence, which is given by
	$$\rho_a(I) = \sup \left\lbrace \frac{m}{r} \, | \, I^{(mt)} \nsubseteq I^{rt} \textrm{ for all } t \gg 0 \right\rbrace$$	
	and was also first defined in \cite{AsymptoticResurgence}. While $\rho_a(I) \leqslant \rho(I)$ always holds, these two invariants can actually differ, as shown by the examples in \cite{DHNSST2015,ResurgenceKleinWiman} and \cite[Section 3]{MonAsymptoticResurgence}. In principle, showing that $\rho_a(I) < h$ might not be enough to answer Question \ref{question any C}, since there are a priori arbitrarily large $a$ and $b$ for which the value of $\rho_a(I)$ might not determine whether $I^{(a)} \subseteq I^b$. However, it has now been shown in \cite[2.11]{GrifoHunekeMukundan} (and also as a consequence of \cite[Proposition 2.6.]{DiPasqualeDrabkin}) that if $\rho_a(I)<h$, then Question \ref{question any C} has an affirmative answer.
	\end{remark}

The answer to Question \ref{question any C} would then be affirmative for all ideals $I$ as long as $\rho(I) < h$. We are not aware of any examples of ideals $I$ whose resurgence and big height coincide; the first natural place to look for such examples would be the known counterexamples to Harbourne's Conjecture \ref{Harbourne's Conjecture}.

\begin{example}
	The Fermat configurations \cite{counterexamples} are known counterexamples to $I^{(3)} \subseteq I^2$ and have resurgence $3/2$ \cite[Theorem 2.1]{DHNSST2015}. Therefore, the containment in Harbourne's Conjecture, $I^{(2n-1)} \subseteq I^n$, holds for all $n \neq 2$. We can also obtain, for example, that $I^{(2n-2)} \subseteq I^n$ holds for all $n > 4$. Note that while the lower bound of $3/2$ for the resurgence follows directly from $I^{(3)} \nsubseteq I^2$, the upper bound $3/2$ was obtained in \cite[Theorem 2.1]{DHNSST2015} without any explicit (non-)containment calculations.
\end{example}

There are other approaches to Question \ref{question any C} one could try; in the spirit of Theorem \ref{thm stable}, we ask the following refinement of Question \ref{question any C}:

\begin{question}\label{question given one any C}
	Let $I$ be a radical ideal of big height $h$ in a regular ring $R$. Fix an integer $C > 0$. If $I^{(hm-C)} \subseteq I^m$ holds for some $m$, does
	$$I^{(hn-C)} \subseteq I^n$$
	hold for all $n \geqslant m$?
\end{question} 

We will see that the answer to this question is yes under additional assumptions.

\begin{discussion}\label{remark n+h}
	Given containments such as $I^{(hn)} \subseteq I^n$, one may ask if $I^{(n+h)} \subseteq I I^{(n)}$ for all $n \gg 0$. This containment cannot hold for all $n \geqslant 1$ and all ideals $I$, since for $n = 1$ and $h = 2$ the statement is $I^{(3)} \subseteq I^2$, which is known to sometimes fail \cite{counterexamples}. Unfortunately, as we will see in Example \ref{Alexandra's Fermat example}, $I^{(n+h)} \subseteq I I^{(n)}$ might fail even for infinitely many values of $n$. 
	
	Nevertheless, if $I$ is such that $I^{(n+h)} \subseteq I I^{(n)}$ does hold for all $n \geqslant 1$, then not only does $I$ satisfy Harbourne's Conjecture \ref{Harbourne's Conjecture}, but the answer to Question \ref{question given one any C} is also affirmative.
	
	To see that Question \ref{question given one any C} has a positive answer under the assumption that $I^{(n+h)} \subseteq I I^{(n)}$ holds for all $n \geqslant 1$, fix $C$ and $m$ such that $I^{(hm-C)} \subseteq I^m$. Then for all $k \geqslant m$, by applying the formula $I^{(n+h)} \subseteq I I^{(n)}$ repeatedly we obtain
	$$I^{(hk-C)} = I^{(hm + h(k-m)-C)} \subseteq I I^{(hm - C + h(k-m-1))} \subseteq \cdots \subseteq I^{k-m} I^{(hm-C)}.$$
	By assumption, $I^{(hm-C)} \subseteq I^m$, and therefore
	$$I^{(hk-C)} \subseteq I^{k-m} I^{(hm-C)} \subseteq I^{k}.$$
	
	We have thus shown that if $I^{(n+h)} \subseteq I I^{(n)}$ does hold for all $n \geqslant 1$, then given $C$ and $m$ such that $I^{(hm-C)} \subseteq I^m$, the containment $I^{(hk-C)} \subseteq I^k$ holds for all $k \geqslant m$. Now suppose that $I$ is such that $I^{(n+h)} \subseteq I I^{(n)}$ holds for all $n \geqslant 1$. Then Harbourne's Conjecture \ref{Harbourne's Conjecture} also holds. Indeed, by the property we just showed, Harbourne's Conjecture \ref{Harbourne's Conjecture} follows by taking $C = h-1$ and $m = 1$ as long as we can show that $I^{(h+1)} \subseteq I^2$ does hold. But this is precisely what we obtain from taking $n=1$ in $I^{(n+h)} \subseteq I I^{(n)}$. 
\end{discussion}

What ideals $I$ do satisfy $I^{(n+h)} \subseteq I I^{(n)}$ for all $n \geqslant 1$? We will see that in characteristic $p$ this holds if the ring $R/I$ is F-pure. In particular, it holds for monomial ideals, ideals defining Veronese rings or determinantal rings, and other classes. Recall that a ring $A$ of prime characteristic $p$ is F-pure \cite{HoRo} if the Frobenius homomorphism $F(a) = a^p$ is pure, meaning that $F \otimes 1 : A \otimes_A A \longrightarrow A \otimes_A M$ is injective for all $A$-modules $M$.

The following result follows the same proof technique as \cite[Theorem 3.2]{GrifoHuneke}, which follows as its corollary; however, we will see that the result we present here has other consequences.

\begin{theorem}\label{thm Fpure stable}
	Let $R$ be a regular ring of characteristic $p>0$. Let $I$ be an ideal in $R$ such that $R/I$ is an F-pure ring, and let $h$ be the big height of $I$. Then for all $n \geqslant 1$,
	$$I^{(n+h)} \subseteq I I^{(n)}.$$
	In particular, if $I^{(hk-C)} \subseteq I^k$ for some $k$ and $C$, then $I^{(hn-C)} \subseteq I^n$ for all $n \geqslant k$.
\end{theorem}
	
\begin{proof}
First, note that we can reduce to the local case. Indeed, the big height of an ideal does not increase under localization, and all localizations of an F-pure ring are F-pure \cite[6.2]{HoRo}; moreover, containments are local statements. So let $(R, \m)$ be a regular local ring.

Fix $n \geqslant 1$. As in \cite[Lemma 3.1]{GrifoHuneke}, we will first show that 
$$\left( I^{[q]} : I \right) \subseteq \left( I I^{(n)} : I^{(n+h)} \right)^{[q]}$$
for all $q = p^e \gg 0$. Once we establish this fact, if $I^{(n+h)} \nsubseteq I I^{(n)}$, then $\left( I I^{(n)} : I^{(n+h)} \right)^{[q]} \subseteq \m^{[q]}$, and thus $\left( I^{[q]} : I \right) \subseteq \mathfrak{m}^{[q]}$, so that by Fedder's Criterion \cite{Fedder} $R/I$ cannot be F-pure.

To show that 
$$\left( I^{[q]} : I \right) \subseteq \left( I I^{(n)} : I^{(n+h)} \right)^{[q]}$$
for all large $q$, consider $s \in \left( I^{[q]} : I \right)$. Then $s I^{(n+h)} \subseteq sI \subseteq I^{[q]}$, and thus
$$s \left( I^{(n+h)} \right)^{[q]} \subseteq \left( s I^{(n+h)} \right) \left( I^{(n+h)} \right)^{q-1} \subseteq I^{[q]} \left( I^{(n+h)} \right)^{q-1}.$$
We will show that
$$\left( I^{(n+h)} \right)^{q-1} \subseteq \left( I^{(n)} \right)^{[q]},$$
which implies that
$$s \left( I^{(n+h)} \right)^{[q]} \subseteq \left( I I^{(n)} \right)^{[q]}.$$

To do that, we will use \cite[Lemma 2.6]{GrifoHuneke}, which says that for all $q = p^e$ we have
$$I^{(hq+(n-1)q-h+1)} \subseteq \left( I^{(n)} \right)^{[q]}.$$

We claim that for all $q \gg 0$, $\left( I^{(n+h)} \right)^{q-1} \subseteq I^{(hq+(n-1)q-h+1)}$, which would conclude the proof that $\left( I^{(n+h)} \right)^{q-1} \subseteq \left( I^{(n)} \right)^{[q]}$.
To show that the claim, it is enough to prove that 
$$(n+h)(q-1) \geqslant hq+(n-1)q-h+1$$
for large values of $q$, and this inequality holds as long as $q \geqslant n+1$.
\end{proof}

On the other hand, the condition $I^{(n+h)} \subseteq I I^{(n)}$ for $n \gg 0$ is strictly stronger than the condition $I^{(hn-C)}~\subseteq~I^n$ for $n \gg 0$, as the following example by Alexandra Seceleanu shows.

\begin{example}[Seceleanu]\label{Alexandra's Fermat example}
	Consider $I = \left(x(y^3-z^3), y (z^3-x^3), z(x^3-y^3) \right) \subseteq \mathbb{C} [x,y,z]$, which defines a reduced set of $12$ points in $\mathbb{P}^2$ known as the Fermat configuration. This was the first counterexample found to $I^{(3)} \subseteq I^2$ \cite{counterexamples}. Writing $f = y^3-z^3$, $g = z^3-x^3$ and $h = x^3-y^3$, we have $H = fgh \in I^{(3)}$, and thus
	$$H^{n+1} \in \left( I^{(3)} \right)^{n+1} \subseteq I^{(3n+3)}.$$
	We will show, however, that $H^{n+1} \notin I I^{(3n+1)}$. To do this, we will compute the degree $9n+9$ part of $I I^{(3n+1)}$, following the same argument used in \cite[Example 3.8]{NguyenVu} to show that the minimal degree of an element in $I^{(3n+1)}$ is $9n+4$, and that 
	$$\left( I^{(3n+1)} \right)_{9n+4} = H^n I_4.$$
	In fact, it is shown in \cite[Example 3.8]{NguyenVu} that
	$$\left( I^{(3n+1)} \right)_{d} = H^n I_{d-9n}$$
	holds for all $d \leqslant 9n+4$, but note that the exact same argument given in \cite{NguyenVu} follows for $d \leqslant 12n + 3$. Now since $I$ is generated in degree $\geqslant 4$,
	$$\left( I I^{(3n+1)} \right)_{9n+9} = I_5 \left( I^{(3n+1)} \right)_{9n+4} + I_4 \left( I^{(3n+1)} \right)_{9n+5} = H^n I_4 I_5.$$
		If $H^{n+1} \in I I^{(3n+1)}$, it must be that $H \in I_4 I_5 \subseteq I^2$, which is false by \cite{counterexamples}.
		
	We have thus shown that $I^{(3n+3)} \nsubseteq I I^{(3n+1)}$ for all $n \geqslant 0$, so $I$ fails $I^{(k+2)} \subseteq I I^{(k)}$ for infinitely many $k$. In contrast, $I^{(2n-C)} \subseteq I^n$ for all $C>0$ and all $n \gg 0$, by Remark \ref{remark any C}.
\end{example}

\section{A special case in dimension $3$}\label{section dim 3 matrix conditions}

We now shift gears from large values to small values, and focus on the  containment $I^{(3)} \subseteq I^2$. Following \cite{Seceleanu}, we will study this question in dimension $3$ for height $2$ quasi-homogeneous ideals generated by the $2 \times 2$ minors of a $2 \times 3$ matrix. This class of ideals is of interest because it includes includes ideals that satisfy $I^{(3)} \subseteq I^2$, and others that fail this containment. In fact, this class contains prime ideals such as those defining space monomial curves $(t^a, t^b, t^c)$ \cite{Herzog1970}, but also known counterexamples to $I^{(3)} \subseteq I^2$ such as the Fermat configurations \cite{counterexamples}. All the results in this section are extensions of the results in \cite{Seceleanu}, with small adjustments in the proofs when necessary. In Section \ref{section space monomial curves} we will apply the results in this section to the case of space monomial curves.

Let us fix some notation to be used through the remainder of the paper. Let $k$ be a field, $R = k \llbracket x,y,z \rrbracket$ or $R = k [x,y,z]$, and $\m = (x,y,z)$. Let $I = I_2(M)$ be the ideal in $R$ generated by the $2 \times 2$ minors of
$$M = \begin{pmatrix} a_1 & a_2 & a_3 \\ b_1 & b_2 & b_3 \end{pmatrix}$$
for some $a_1, a_2, a_3, b_1, b_2, b_3 \in R$. If $R = k[x,y,z]$, assume that $I$ is quasi-homogeneous.
Write
	$$f_1 = a_2b_3-a_3b_2, \qquad f_2 = a_3b_1-a_1b_3 \quad \textrm{ and } \quad f_3 = a_1b_2 - a_2b_1,$$
so that $I = \left( f_1, f_2, f_3 \right)$.

\begin{proposition}[Proposition 3.1 in \cite{Seceleanu}]\label{proposition ext}
	Consider integers $a > b$, and let $\iota$ be the natural inclusion map $I^a \subseteq I^b$. The following are equivalent:
	\begin{enumerate}[(a)]
		\item $I^{(a)} \subseteq I^b$;
		\item The map $\Ext^{2}_R(\iota): \Ext^{2}_R \left( I^b, R \right) \longrightarrow \Ext^{2}_R \left( I^a, R \right)$ is the zero map.
	\end{enumerate}
\end{proposition}

\begin{remark}
	In general, for a Gorenstein local ring $(R, \m)$ of dimension $d$, $\left( I^a \right)^{\textrm{sat}} \subseteq I^b$ if and only if $\Ext^{d-1}_R(\iota) = 0$, where $J^{\textrm{sat}} = \left( J : \m^\infty \right)$ denotes the saturation of $J$ with respect to $\m$. Our conditions give us control over the embedded primes of the powers of $I$; more precisely, they guarantee that $\m$ is the only possible embedded prime of any power of $I$, and the same would follow for homogeneous saturated ideals of height $d-1$ in any regular local ring of dimension $d$. In all these cases, $\left( I^a \right)^{\textrm{sat}} = I^{(a)}$ for all $a$. However, dimension $3$ is the only circumstance where we can apply this technique to all (quasihomogeneous) radical ideals of height $2$.
\end{remark}

Seceleanu \cite{Seceleanu} proceeds by computing this Ext map explicitly when $a=3$ and $b=2$. Her method also works for any $a$ and $b$. To do this, one extracts minimal free resolutions for all powers of $I$ from a resolution of the Rees algebra of $I$, which is the graded algebra
	$$\mathcal{R}(I) = \bigoplus_{n \geqslant 0} I^n t^n \subseteq R[t],$$
	whose degree $n$ piece is isomorphic to $I^n$. We can view the Rees algebra of $I$ as a quotient of the polynomial ring $S = R[T_1, T_2, T_3]$, where elements in $R$ have degree $0$ and each $T_i$ has degree $1$, via the graded map $S \longrightarrow \ra(I)$ given by $T_i \mapsto f_it$. Via this map, we can write $\ra(I) \cong S/L$, where
	$$L = \left( F(T_1, T_2, T_3 ) \in R \left[ T_1, T_2, T_3 \right] \,\, \big| \,\, F(f_1, f_2, f_3) = 0 \right).$$

	Recall that the symmetric algebra of $I$ is the quotient of the tensor algebra $\bigoplus_{n \geqslant 0} I^{\otimes n}$ by the ideal generated by the simple tensors of the form $u \otimes v - v \otimes u$. The symmetric algebra of $I$ can also be written as a quotient of $S$, as $\sym(I) \cong S/L_1$ with
	$$L_1 = \left( a_1 T_1 + a_2 T_2 + a_3 T_3 \,\,\, \bigg| \,\, a_1 f_1 + a_2 f_3 + a_3 f_3 = 0 \, \right).$$
	Notice that $L_1$ coincides with the ideal generated by the degree $1$ part of $L$ above. Moreover, the multiplication map on $I \otimes I$ induces a surjective graded map $\sym(I) \rightarrow \ra(I)$. We say that an ideal $I$ is \emph{of linear type} if the map $\sym(I) \rightarrow \ra(I)$ is an isomorphism.
	
\begin{lemma}[Lemma 2.1 and Lemma 2.2 in \cite{Seceleanu}]\label{resolution of the rees algebra}	

	The ideal $I$ is of linear type, and the Rees algebra of $I$ is a complete intersection over $S$, with minimal free resolution given by
	$$\xymatrix{0 \ar[r] & S(-2) \ar[r]^-{\begin{bmatrix} F \\ G \end{bmatrix}} & S(-1) \oplus S(-1) \ar[rr]^-{\begin{bmatrix} G & -F \end{bmatrix}} && S \ar[r] & \ra(I) \ar[r] & 0},$$
	where
	\begin{align*}
		F = a_1 T_1 + a_2 T_2 + a_3 T_3 & &\textrm{and} && & G = b_1 T_1 + b_2 T_2 + b_3 T_3.
	\end{align*}
\end{lemma}

\begin{proof}
	By \cite{HunekeReesAlgdsequence}, ideals generated by $d$-sequences are of linear type, and $I$ is generated by a $d$-sequence. Our assumption that $I$ has height $2$ implies in particular that the two rows of $M$ are linearly independent, so that $F$, $G$ is a regular sequence in $S$. To show that the Rees algebra of $I$ is a complete intersection, all we need to see is that the kernel $L = L_1$ of $T_i \mapsto f_i t$ is the ideal $(F,G)$. By the Hilbert--Burch Theorem, a minimal free resolution of $I$ is given by
	$$\xymatrix@C=14mm{0 \ar[r] & R^2 \ar[r]^-{M^T} & R^3 \ar[r]^-{\left[f_1 \, f_2 \, f_3 \right]} & I \ar[r] & 0}.$$
	In particular, any relation between the generators of $I$ is an $R$-linear combination of the relations $(a_1, a_2, a_3)$ and $(b_1, b_2, b_3)$, so that $F$ and $G$ minimally generate $L$.\qedhere
\end{proof}

\begin{discussion}[Resolutions for all powers]\label{discussion}
	Since $\ra(I)_n = I^n t^n$, we can now extract resolutions for all powers of $I$ by taking the degree $n$ strand of the resolution for $\ra(I)$ over $S$, as in \cite[Proposition 2.3]{Seceleanu}. Note that $S_n$ is the free $R$-module generated by all $n+2 \choose 2$ monomials of degree $n$ in $T_1$, $T_2$, and $T_3$, so we will identify $S_n$ with $R^{n+2 \choose 2}$. Thus $I^n$ has a free resolution over $R$ given by 
$$\xymatrix{0 \ar[r] & R^{n \choose 2} \ar[r]^-{\varphi(n)} & R^{n+1 \choose 2} \oplus R^{n+1 \choose 2} \ar[r]^-{\psi(n)} & R^{n+2 \choose 2} \ar[r] & I^n \ar[r] & 0}$$
where $\varphi(n)$ is the map induced by multiplication by $F$ on the first copy of $R^{n+1 \choose 2}$ and by multiplication by $G$ on the second copy of $R^{n+1 \choose 2}$. 

More precisely, let $\alpha$ and $\beta$ be generators of the free $S$-module in homological degree $1$ in the minimal free resolution
$$\xymatrix{0 \ar[r] & S(-2) \ar[r]^-\varphi & S(-1) \oplus S(-1) \ar[r]^-\psi & S \ar[r] & \ra(I) \ar[r] & 0},$$
and let $\gamma$ be a generator of the free $S$-module in homological degree $2$. With this notation, $T_1^iT_2^jT_3^k\gamma$, where $i+j+k=n-2$, will denote a generator in homological degree $2$ in
$$\xymatrix{0 \ar[r] & R^{n \choose 2} \ar[r]^-{\varphi(n)} & R^{n+1 \choose 2} \oplus R^{n+1 \choose 2} \ar[r]^-{\psi(n)} & R^{n+2 \choose 2} \ar[r] & I^n \ar[r] & 0},$$
while $T_1^iT_2^jT_3^k\alpha$ and $T_1^iT_2^jT_3^k\beta$, where $i+j+k=n-1$, denote generators in homological degree $1$. With this notation, given $i+j+k=n-2$, the map $\varphi(n)$ takes $T_1^iT_2^jT_3^k\gamma$ to
$$\left( a_1 T_1^{i+1}T_2^jT_3^k + a_2 T_1^iT_2^{j+1}T_3^k + a_3 T_1^iT_2^jT_3^{k+1} \right)\alpha + \left( b_1 T_1^{i+1}T_2^jT_3^k + b_2 T_1^iT_2^{j+1}T_3^k + b_3 T_1^iT_2^jT_3^{k+1}\right)\beta.$$
If we represent $\varphi(n)$ as a matrix with $n-2$ columns and $2(n-1)$ rows, the entry in the column corresponding to $T_1^iT_2^jT_3^k\gamma$, where $i+j+k=n-2$, and the row corresponding to $T_1^uT_2^wT_3^v\delta$, where $u+w+v=n-1$ and $\delta =\alpha$ or $\delta = \beta$, is
\begin{itemize}
	\item $a_1$ if $\delta=\alpha$, $u=i+1$, $w=j$ and $v=k$;
	\item $a_2$ if $\delta=\alpha$, $u=i$, $w=j+1$ and $v=k$;
	\item $a_3$ if $\delta=\alpha$, $u=i$, $w=j$ and $v=k+1$;
	\item $b_1$ if $\delta=\beta$, $u=i+1$, $w=j$ and $v=k$;
	\item $b_2$ if $\delta=\beta$, $u=i$, $w=j+1$ and $v=k$;
	\item $b_3$ if $\delta=\beta$, $u=i$, $w=j$ and $v=k+1$.
\end{itemize}
\end{discussion}
	
	\begin{example}
	We can now recover the matrix $\varphi(3)^T$ that appears in \cite{Seceleanu}: 
	$$\begin{blockarray}{lcccccccccccc} 
	& \blue{T_1^2\alpha} & \blue{T_1T_2\alpha} & \blue{T_1T_3\alpha} & \blue{T_2^2\alpha} & \blue{T_2T_3\alpha} & \blue{T_3^2\alpha} & \blue{T_1^2\beta} & \blue{T_1T_2\beta} & \blue{T_1T_3\beta} & \blue{T_2^2\beta} & \blue{T_2T_3\beta} & \blue{T_3^2\beta} \\
	\begin{block}{l(cccccccccccc)}
	\blue{T_1\gamma} & a_1 & a_2 & a_3 & 0 & 0 & 0 & b_1 & b_2 & b_3 & 0 & 0 & 0 \\ 
	\blue{T_2\gamma} & 0 & a_1 & 0 & a_2 & a_3 & 0 & 0 & b_1 & 0 & b_2 & b_3 & 0 \\ 
	\blue{T_3\gamma} & 0 & 0 & a_1 & 0 & a_2 & a_3 & 0 & 0 & b_1 & 0 & b_2 & b_3 \\
	\end{block}
	\end{blockarray}.$$
	\end{example}
	
	\begin{remark}
	There are various other ways to obtain these minimal free resolutions of each $I^n$, such as \cite{Weyman,SymbPowersCodim2}. These resolutions also correspond to the Z-complex in the approximation complex construction \cite{HerzogSimisVasconcelosI,HerzogSimisVasconcelosII}, which was pointed out to the author by Vivek Mukundan.
	The advantage of the method we use here, following Seceleanu, is that we will also be able to easily extract lifts of the maps $I^{n+1} \subseteq I^n$, which we will then use to compute $\Ext(I^{a} \subseteq I^b, R)$ for all $a > b$.
	\end{remark}

	We want to compute $\Ext(I^{a} \subseteq I^b, R)$, so it remains to find lifts for all inclusion maps $I^n \subseteq I^{n-1}$ that are compatible with these resolutions. Consider the additive maps on $S$ given by
	\begin{align*}
D_1 (T_1^i T_2^j T_3^k) = 
	\left\lbrace \begin{array}{ll} T_1^{i-1} T_2^j T_3^k & \textrm{if } i \geqslant 1 \\ 0 & \textrm{otherwise}
	\end{array}\right. \\
	D_2 (T_1^i T_2^j T_3^k) = 
	\left\lbrace \begin{array}{ll} T_1^{i} T_2^{j-1} T_3^k & \textrm{if } j \geqslant 1 \\ 0 & \textrm{otherwise}
	\end{array}\right.  \\
	D_3 (T_1^i T_2^j T_3^k) = 
	\left\lbrace \begin{array}{ll} T_1^{i} T_2^j T_3^{k-1} & \textrm{if } k \geqslant 1 \\ 0 & \textrm{otherwise}
	\end{array}\right. .		
	\end{align*}

\begin{lemma}[Proposition 3.2 in \cite{Seceleanu}]
	The operator 
	$$D = f_1 D_1 + f_2 D_2 + f_3 D_3$$
	on $S$ induces the degree $-1$ map on $\ra(I)$ that takes an homogeneous element $gt^n \in I^nt^n$ to $3gt^{n-1} \in I^{n-1} t^{n-1}$. 
	
	Moreover, writing $D_n$ to represent the map $S_n \longrightarrow S_{n-1}$ induced by $D$, the following is a commutative diagram with exact rows:
	$$\xymatrix{0 \ar[r] & R^{n-1 \choose 2} \ar[r]^-{\varphi(n-1)} & R^{n \choose 2} \oplus R^{n \choose 2} \ar[r]^-{\psi(n-1)} & R^{n+1 \choose 2} \ar[r] & I^{n-1} \ar[r] & 0\\
	0 \ar[r] & R^{n \choose 2} \ar[u]^-{D_{n-2}} \ar[r]^-{\varphi(n)} & R^{n+1 \choose 2} \oplus R^{n+1 \choose 2} \ar[u]_-{D_{n-1}} \ar[r]^-{\psi(n)} & R^{n+2 \choose 2} \ar[u]_-{D_n} \ar[r] & I^n \ar[u]_-{3 \, \iota} \ar[r] & 0}$$
\end{lemma}

\begin{proof}
Since 
$$D(F)=a_1f_1 + a_2 f_2 + a_3 f_3 = 0 = b_1 f_1 + b_2 f_2 + b_3 f_3 = D(G),$$
$D$ induces a map on $S/(F,G) \cong \ra(I)$. To check that the induced map is as claimed, it is enough to show the claim for elements of the form $gt^n = f_1^if_2^jf_3^kt^n$, where $i+j+k = n$. Note that such an element is the image of $T_1^iT_2^jT_3^k$ via the surjection $S \longrightarrow \ra(I)$, and that
$$D \left( T_1^iT_2^jT_3^k \right) = f_1T_1^{i-1}T_2^jT_3^k + f_2T_1^iT_2^{j-1}T_3^k + f_3 T_1^iT_2^jT_3^{k-1}.$$
Now the map $S \longrightarrow \ra(I)$ takes this element to
$$\left( f_1f_1^{i-1}f_2^jf_3^k + f_2f_1^if_2^{j-1}f_3^k + f_3 f_1^if_2^jf_3^{k-1} \right) t^{n-1} = 3gt^{n-1}.$$
	To check that the diagram commutes, we need to check that $D$ commutes with the maps $[F \, G]$ and $[G \, -F]^T$ in the minimal free resolution of $\ra(I)$. And indeed,
	\begin{align*}
		D \left( F \cdot T_1^i T_2^j T_3^k \right) & = D \left( a_1 T_1^{i+1} T_2^j T_3^k + a_2 T_1^i T_2^{j+1} T_3^k + a_3 T_1^i T_2^j T_3^{k + 1} \right) \\
		& =  a_1 \left( f_1 T_1^{i} T_2^j T_3^k + f_2 T_1^{i+1} T_2^{j-1} T_3^k + f_3 T_1^{i+1} T_2^{j} T_3^{k-1} \right) \\
		& + a_2 \left( f_1 T_1^{i-1} T_2^{j+1} T_3^{k} + f_2 T_1^i T_2^{j} T_3^k + f_3 T_1^{i+1} T_2^{j} T_3^{k-1} \right) \\
		& + a_3 \left( f_1 T_1^{i-1} T_2^{j} T_3^{k+1} + f_2 T_1^i T_2^{j-1} T_3^{k +1} + f_3 T_1^{i} T_2^{j} T_3^{k} \right) \\
		& = a_1 T_1 \left( f_1 T_1^{i-1} T_2^j T_3^k + f_2 T_1^{i} T_2^{j-1} T_3^k + f_3 T_1^{i} T_2^{j} T_3^{k-1} \right) \\
		& + a_2 T_2 \left( f_1 T_1^{i-1} T_2^{j} T_3^{k} + f_2 T_1^i T_2^{j-1} T_3^k + f_3 T_1^{i+1} T_2^{j-1} T_3^{k-1} \right) \\
		& + a_3 T_3 \left( f_1 T_1^{i-1} T_2^{j} T_3^{k} + f_2 T_1^i T_2^{j-1} T_3^{k} + f_3 T_1^{i} T_2^{j} T_3^{k-1} \right) \\
		& = F \cdot \left( f_1 T_1^{i-1} T_2^{j} T_3^{k} + f_2 T_1^i T_2^{j-1} T_3^k + f_3 T_1^{i+1} T_2^{j-1} T_3^{k-1} \right) \\ 
		& = F \cdot D \left( T_1^{i} T_2^{j} T_3^{k} \right).
	\end{align*}
\end{proof}

Given any $n > m$, we compose successive commutative diagrams of this type, and obtain
	
	$$\xymatrix{0 \ar[r] & R^{m \choose 2} \ar[r]^-{\varphi(m)} & R^{m+1 \choose 2} \oplus R^{m+1 \choose 2} \ar[r]^-{\psi(m)} & R^{m+2 \choose 2} \ar[r] & I^{m} \ar[r] & 0\\
	0 \ar[r] & R^{n \choose 2} \ar[u]_-{D_{n,m}} \ar[r]_-{\varphi(n)} & R^{n+1 \choose 2} \oplus R^{n+1 \choose 2} \ar[u] \ar[r]_-{\psi(n)} & R^{n+2 \choose 2} \ar[u] \ar[r] & I^n \ar[u]_-{3 \iota} \ar[r] & 0}$$
	where $D_{n,m} := D_{n-2} D_{n-3} \cdots D_{m} D_{m-1}$ and $\iota$ denotes the inclusion map $I^n \subseteq I^m$.

	\begin{discussion}
	To write down the map $D_{n,m}\!: R^{n \choose 2} \rightarrow R^{m \choose 2} $, we think of each copy of $R$ in $R^{n \choose 2}$ as corresponding to one of the monomials $T_1^iT_2^jT_3^k$ of degree $i+j+k= n-2$; the set of all such monomials is a basis for $R[T_1,T_2,T_3]_{n-2}$. Note that the maps $D_1$, $D_2$ and $D_3$ commute with each other, and thus we can rewrite $D_{n,m}$ as
	$$D_{n,m} = ( f_1 D_1 + f_2 D_2 + f_3 D_3)^{n-m} = \sum_{u+v+w=n-m}  \frac{(n-m)!}{u!v!w!} \, f_1^{u} f_2^v f_3^w  D_1^{u} D_2^v D_3^w.$$
	The coefficients count all the ways to order $u$ copies of $D_1$, $v$ copies of $D_2$ and $w$ copies of $D_3$.
		We conclude that
		$$D_{n,m} \left( T_1^i T_2^j T_3^k \right) = \sum_{u+v+w = n-m} \frac{(n-m)!}{u!v!w!} \, f_1^{u} f_2^v f_3^w T_1^{i-u} T_2^{j-v} T_3^{k-w}.$$
		Finally, $V_{n,m}$ is simply the transpose of this map. Given $i+j+k = m$,  $V_{n,m}(T_1^i T_2^v T_3^w)$ is a combination of $T_1^{i+u} T_2^{j+v} T_3^{k+w}$ for each $u+v+w = n-m$, with coefficients as above, i.e.,
		$$V_{n,m}(T_1^i T_2^v T_3^w) = \sum_{u+v+w = n-m} \frac{(n-m)!}{u!v!w!} \, f_1^{u} f_2^v f_3^w T_1^{i+u} T_2^{j+v} T_3^{k+w}.$$
	\end{discussion}
	
	\begin{corollary}[Proposition 3.2 in \cite{Seceleanu}]\label{maps lemma 3}
	Let $n > m$ and suppose that $\ch k \neq 3$. Then $I^{(n)} \subseteq I^m$ if and only if the columns of $V_{n,m} := \left( D_{n,m} \right)^T$ are contained in the image of $H_n = \left(\varphi(n)\right)^T$.
	\end{corollary}

	The following provides a convenient way to simplify computations.

	\begin{theorem}\label{matrix main result}
	Let $n > m \geqslant 1$, and suppose that $\ch k \neq 3$. Then $I^{(n)} \subseteq I^m$ if and only if for each $i$, $j$, and $k$ such that $i+j+k = m-2$, there exist a choice of $u$, $v$ and $w$ such that $u+v+w = n-m$ and
	$$f_1^u f_2^v f_3^w T_1^{i+u}T_2^{j+v}T_3^{k+w}$$
	is an element of the ideal generated by
	$$\left\lbrace \begin{array}{c} a_1 T_1^{d-1} T_2^{e} T_3^f + a_2 T_1^{d} T_2^{e-1} T_3^f + a_3 T_1^{d} T_2^{e} T_3^{f-1}, \\ b_1 T_1^{d-1} T_2^{e} T_3^f + b_2 T_1^{d} T_2^{e-1} T_3^f + b_3 T_1^{d} T_2^{e} T_3^{f-1} \end{array}  \right\rbrace_{d+e+f = n-1}.$$
\end{theorem}

\begin{proof}
	For fixed $i$, $j$ and $k$, we will show that given two choices of $u+v+w = u'+v'+w'=n-m$, $f_1^u f_2^v f_3^w T_1^{i+u}T_2^{j+v}T_3^{k+w}$ and $f_1^{u'} f_2^{v'} f_3^w T_1^{i+u'}T_2^{j+v'}T_3^{k+w'}$ differ by an element in the image of $H_n$. The theorem will follow from this claim. To see that, note first that the sum of the coefficients appearing in $V_{n,m}(T_1^iT_2^jT_3^k)$ is
	$$\sum_{u+v+w = n-m} \frac{(n-m)!}{u!v!w!} = (1+1+1)^{n-m} = 3^{n-m},$$
	As a consequence, the difference between
	$$3^{n-m} f_1^u f_2^v f_3^w T_1^{i+u}T_2^{j+v}T_3^{k+w}$$
	and
	$$V_{n,m} \left( T_1^iT_2^jT_3^k \right) = \sum_{u+v+w = n-m} \frac{(n-m)!}{u!v!w!} \, f_1^{u} f_2^v f_3^w T_1^{i+u} T_2^{j+v} T_3^{k+w}$$
	is in the image of $H_n$. Since $I^{(n)} \subseteq I^m$ is equivalent to $\im V_{n,m} \subseteq \im H_n$, the theorem will follow once we prove the claim above. Note that we can drop the $3^{n-m}$ coefficient because $3$ is invertible by assumption.
	
	To show our claim, we will first see that given any monomial $P(T)$ in $T_1$, $T_2$, and $T_3$ of degree $n-3$,
	$$\left( f_c T_c - f_dT_d \right) P(T) \gamma \in \im H_n$$
	for any $c \neq d$ with $c, d \in \left\lbrace 1, 2, 3 \right\rbrace$. We will show this for $c=1$ and $d=2$, noting that the proof is similar for other values. If $P(T) = T_1^iT_2^jT_3^k$, then
	$$b_3 \left( a_1 T_1^{i} T_2^{k+1} T_3^j + a_2 T_1^{i+1} T_2^{k} T_3^j + a_3 T_1^{i+1} T_2^{k+1} T_3^{j-1} \right)$$
	$$ - a_3 \left( b_1 T_1^{i} T_2^{k+1} T_3^j + b_2 T_1^{i+1} T_2^{k} T_3^j + b_3 T_1^{i+1} T_2^{k+1} T_3^{j-1} \right).$$
	$$= -f_2 T_1^{i} T_2^{k+1} T_3^j + f_1 T_1^{i+1} T_2^{k} T_3^j + 0 \, T_1^{i+1} T_2^{k+1} T_3^{j-1}$$
	$$= \left( -f_2 T_2 + f_1 T_1 \right)P(T).$$
	
	Now given two choices $u+v+w = u'+v'+w' = n-m$, we claimed that
		$$f_1^u f_2^v f_3^w T_1^{i+u}T_2^{j+v}T_3^{k+w} - f_1^{u'} f_2^{v'} f_3^{w'} T_1^{i+u'}T_2^{j+v'}T_3^{k+w'} \in \im H_n.$$
	Let us first prove this when $u'=u+1$, $v'=v-1$, and $w'=w$. We want to show that
	$$f_1^u f_2^v f_3^w T_1^{i+u}T_2^{j+v}T_3^{k+w} - f_1^{u+1} f_2^{v-1} f_3^{w} T_1^{i+u+1}T_2^{j+v-1}T_3^{k+w} \in \im H_n.$$
	And indeed, this can be rewritten as
	$$f_1^u f_2^{v-1} f_3^w \left( f_2 T_1^{i+u}T_2^{j+v}T_3^{k+w} - f_1 T_1^{i+u+1}T_2^{j+v-1}T_3^{k+w} \right) \in \im H_n.$$
	Similar computations give the remaining cases if we switch the roles of $u$, $v$ and $w$. Now an inductive argument shows the claim holds when $|u-u'|>1$, by simply taking successive sums of differences of this sort.
	\end{proof}

\begin{remark}
	Theorem \ref{matrix main result} says that in order to show that $I^{(n)} \subseteq I^m$, it is enough to show that for each column of $V_{n, m}$, the vector obtained by substituting all of the nonzero entries but one by $0$ is in the image of $H_n$. This simplifies our calculations considerably if we want to prove that specific containments hold.
\end{remark}

Applying this to the containment $I^{(3)} \subseteq I^2$, we obtain the following extension of \cite[Theorem 3.3]{Seceleanu}:

\begin{theorem}\label{main result for 23}
	Suppose that $\ch k \neq 3$. The containment $I^{(3)} \subseteq I^2$ holds if and only if one (equivalently, all) of the following vectors
	$$\begin{bmatrix} f_1 \\ f_2 \\ f_3 \end{bmatrix}, \qquad \begin{bmatrix} f_1 \\ 0 \\ 0 \end{bmatrix}, \qquad \begin{bmatrix} 0 \\ f_2 \\ 0 \end{bmatrix}, \quad\textrm{ and } \quad \begin{bmatrix} 0 \\ 0 \\ f_3 \end{bmatrix}
$$
are in the image of 
$$\begin{pmatrix}
	a_1 & a_2 & a_3 & 0 & 0 & 0 & b_1 & b_2 & b_3 & 0 & 0 & 0 \\ 
	0 & a_1 & 0 & a_2 & a_3 & 0 & 0 & b_1 & 0 & b_2 & b_3 & 0 \\ 
	0 & 0 & a_1 & 0 & a_2 & a_3 & 0 & 0 & b_1 & 0 & b_2 & b_3
	\end{pmatrix}.$$
\end{theorem}

\begin{example}
	Suppose we want to show that $I^{(5)} \subseteq I^3$. That would require proving that
	$$\begin{bmatrix} f_1^2 \\ 2f_1f_2 \\ 2f_1f_3 \\ f_2^2 \\ 2f_2f_3 \\ f_3^2 \\ 0 \\ 0 \\ 0 \\ 0 \end{bmatrix}, \qquad \begin{bmatrix} 0 \\ f_1^2 \\ 0 \\ 2f_1f_2 \\ 2f_1f_3 \\ 0 \\ f_2^2 \\ 2f_2f_3 \\ f_3^2 \\ 0 \end{bmatrix},\qquad \begin{bmatrix} 0 \\ 0 \\ f_1^2 \\ 0 \\ 2f_1f_2 \\ 2f_1f_3 \\ 0 \\ f_2^2 \\ 2f_2f_3 \\ f_3^2 \end{bmatrix} \qquad \in \im H_{5}.$$
	Theorem \ref{matrix main result} says that this is in fact equivalent to showing, for instance, that $$\begin{bmatrix} f_1^2 \\ 0 \\ 0 \\ 0 \\ 0 \\ 0 \\ 0 \\ 0 \\ 0 \\ 0 \end{bmatrix}^T, \begin{bmatrix} 0 \\ 0 \\ 0 \\ f_1f_2 \\ 0 \\ 0 \\ 0 \\ 0 \\ 0 \\ 0 \end{bmatrix}^T, \begin{bmatrix} 0 \\ 0 \\ 0 \\ 0 \\ 0 \\ 0 \\ 0 \\ f_2^2 \\ 0 \\ 0 \end{bmatrix}^T \in \im H_5.$$
\end{example}

\begin{remark}\label{remark char 3}
	If we wanted to use this technique in characteristic $3$, we could replace our choice of map $D$ by
	$$f_1 \frac{\partial}{\partial T_1} + f_2 \frac{\partial}{\partial T_2} + f_3 \frac{\partial}{\partial T_3}.$$
	The distinction is that this map will give us lifts of $n \cdot \left( I^{n} \subseteq I^{n-1} \right)$, and instead of excluding characteristic $3$, we will have to exclude all characteristics dividing $\frac{n!}{m!}$ when using this map to study $I^{(n)} \subseteq I^m$. For appropriate choices of $n$ and $m$, this will allow us to say something about $I^{(n)} \subseteq I^m$ in characteristic $3$. In particular, Theorem \ref{matrix main result} also holds for this new choice of lifts. The proof is very similar to the one we present here, the only difference being that the sum of the coefficients is $\frac{n!}{m!}$ instead of $3^{n-m}$. Unfortunately, none of these approaches apply to studying the containment $I^{(3)} \subseteq I^2$ in characteristic $3$.
		
	In fact, this work was originally done \cite[Chapter 3]{mythesis} using the map $f_1 \frac{\partial}{\partial T_1} + f_2 \frac{\partial}{\partial T_2} + f_3 \frac{\partial}{\partial T_3}$. Vincent G\'elinas first suggested to the author to consider the map $D$ we use here instead; conversations with Alexandra Seceleanu were also crucial to make this change.
\end{remark}

\vspace{2em}

In the next and final section, we will show sufficient conditions to imply $I^{(3)} \subseteq I^2$, $I^{(4)} \subseteq I^3$, and $I^{(5)} \subseteq I^3$. For that, we will exhibit explicit solutions to the linear equations we need to solve, by writing combinations of columns of $H_n := \varphi(n)^T$. For the reader's convenience, each column used will be tagged as above, with a monomial $T_1^iT_2^jT_3^k$ of degree $n-1$, but we will drop $\alpha$ and $\beta$ from our notation. A column tagged with $T_1^i T_2^j T_3^k$ has non-zero entries in rows $T_1^{i-1} T_2^j T_3^j$ (if $i>0$), $T_1^{i} T_2^{j-1} T_3^j$ (if $j>0$) and $T_1^{i} T_2^j T_3^{k-1}$ (if $k>0$), as described in Discussion \ref{discussion}.

\vskip 5mm

\section{Space monomial curves}\label{section space monomial curves}

Let $a, b, c$ be positive integers. We will write $P(a,b,c)$ to denote the kernel the map $R = k \llbracket x,y,z \rrbracket \longrightarrow k \llbracket t \rrbracket$ or $R = k [ x,y,z ] \longrightarrow k [ t ]$ defined by $x \mapsto t^a$, $y \mapsto t^b$ and $z \mapsto t^c$. By \cite{Herzog1970}, $P(a,b,c) = I_2(M)$, where $M$ is a matrix of the form
$$M = \begin{pmatrix} x^{\alpha_3} & y^{\beta_1} & z^{\gamma_2} \\ z^{\gamma_1} & x^{\alpha_2} & y^{\beta_3} \end{pmatrix}.$$
This is a rich class of prime ideals, whose symbolic powers exhibit different kinds of behavior. In particular, their symbolic Rees algebras
$$\bigoplus_{n \geqslant 0} I^n t^n \subseteq R[t]$$
are not always finitely generated \cite{NonnoetherianSymb}. Even when they are finitely generated, the symbolic Rees algebra may be generated up to various degrees \cite{MonCurvesGen2,NoethSymbReesAlgDegrees,Degree4SpaceMonCurvesI,Degree4SpaceMonCurvesII}. For some choices of $a, b, c$, the ideal $P(a,b,c)$ is a complete intersection, and thus the containment problem is not very interesting, since $P^{(n)} = P^n$ for all $n \geqslant 1$. The containment problem becomes interesting when $P(a,b,c)$ is not a complete intersection: in that case, we even have $P^{(n)} \neq P^n$ for all $n \geqslant 1$, by \cite{Huneke1986}.

The main theorem in this section is a positive answer to Huneke's Question for these primes:

\begin{theorem}\label{thm space monomial curves}
	Let $\ch k \neq 3$. If $P=P(a,b,c)$ for any $a, b, c$, then $P^{(3)} \subseteq P^2$.
\end{theorem}

When $\ch k = 2$, this is a simple consequence of \cite{comparison}. Otherwise, Theorem \ref{thm space monomial curves} is a corollary of the following more general criterion:

\begin{theorem}\label{sufficient condition for 23}
	Suppose that $\ch k \neq 2, 3$. If $a_1| b_2a_3$, then $I^{(3)} \subseteq I^2$.
\end{theorem}

\begin{proof}
	Suppose that $b_2 a_3 = c a_1$, and recall that $f_1 = a_2b_3 - a_3 b_2$. Then
	$$\begin{bmatrix} a_2b_3 - a_3b_2 \\ 0 \\ 0 \end{bmatrix} = a_2 \, \myvector{T_1T_3}{b_3 \\ 0 \\ b_1} - b_1 \, \myvector{T_2 T_3}{0 \\ a_3 \\ a_2} + a_3 \, \myvector{T_1T_2}{b_2 \\ b_1 \\ 0} - 2 c \, \myvector{T_1^2}{ a_1 \\ 0 \\ 0}.\qedhere$$
\end{proof}

\begin{remark}
	To see that Theorem \ref{thm space monomial curves} follows from Theorem \ref{sufficient condition for 23}, note that by \cite{Herzog1970},
$$P(a,b,c) = I_2 \begin{pmatrix} x^{\alpha_3} & y^{\beta_1} & z^{\gamma_2} \\ z^{\gamma_1} & x^{\alpha_2} & y^{\beta_3} \end{pmatrix}$$
for some $\alpha_i, \beta_i, \gamma_i$. If $\alpha_3 \leqslant \alpha_2$, this matrix satisfies $a_1 | b_2$; otherwise, row and column operations give
$$P(a,b,c) = I_2 \begin{pmatrix} x^{\alpha_3} & y^{\beta_1} & z^{\gamma_2} \\ z^{\gamma_1} & x^{\alpha_2} & y^{\beta_3} \end{pmatrix} = I_2 \begin{pmatrix}  x^{\alpha_2} & z^{\gamma_1} & y^{\beta_3} \\ y^{\beta_1} & x^{\alpha_3} &z^{\gamma_2} \end{pmatrix},$$
which is of the desired form.
\end{remark}

Similarly, we can show the next instance of Harbourne's Conjecture holds for space monomial curves.

\begin{theorem}\label{sufficient condition for 35}
	Suppose that $\ch k \neq 2$. If $a_1 | b_2$ and $a_2 | b_3$, then $I^{(5)} \subseteq I^3$. In particular, $P(a,b,c)^{(5)} \subseteq P(a,b,c)^3$ for all $a, b, c$.\end{theorem}

\begin{proof}
	When $\ch k = 3$, $I^{(5)} \subseteq I^3$ holds for ideals of big height $2$ by \cite{comparison}, so we may assume that $\ch k \neq 3$. Write $b_2 = c a_1$ and $b_3 = d a_2$. 
	Then 
	$$f_2^2 = (a_3b_1 - a_1 b_3)^2 = (a_3b_1 - a_1 a_2 d)^2 = a_3^2b_1^2 - 2 a_1 a_2 a_3 b_1 d + a_1^2 a_2^2 d^2, \textrm{ and}$$
\begin{align*}
\begin{bmatrix} 0 \\ 0 \\ 0 \\ 0 \\ 0 \\ 0 \\ 3 f_2^2 \\ 0 \\ 0 \\ 0 \end{bmatrix} = - a_3^2 b_1 c \, \myvector{T_1^2T_2^2}{0 \\ a_2 \\ 0 \\ a_1 \\ 0 \\ 0 \\ 0 \\ 0 \\ 0 \\ 0} + a_2a_3b_1c \, \myvector{T_1^2 T_2 T_3}{0 \\ a_3 \\ a_2 \\ 0 \\ a_1 \\ 0 \\ 0 \\ 0 \\ 0 \\ 0} - a_2^2b_1c \,\myvector{T_1^2 T_3}{0 \\ 0 \\ a_3 \\ 0 \\ 0 \\ a_1 \\ 0 \\ 0 \\ 0 \\ 0} + 2a_2 a_3 b_1 d \, \myvector{T_1 T_2^2}{0 \\ 0 \\ 0 \\ a_2 \\ 0 \\ 0 \\ a_1 \\ 0 \\ 0 \\ 0}
\\
- a_2^2 b_1 d \,\myvector{T_1 T_2^2 T_3}{0 \\ 0 \\ 0 \\ a_3 \\ a_2 \\ 0 \\ 0 \\ a_1 \\ 0 \\ 0} + \begin{pmatrix} 3a_1^2 a_2 d^2 \\ - 9 a_1 a_3 b_1 d \end{pmatrix} \, \myvector{T_2^3}{0 \\ 0 \\ 0 \\ 0 \\ 0 \\ 0 \\ a_2 \\ 0 \\ 0 \\ 0} + \begin{pmatrix} a_1 a_2 b_1 d \\ + 2 a_3 b_1^2 \end{pmatrix} \myvector{T_2^3T_3}{0 \\ 0 \\ 0 \\ 0 \\ 0 \\ 0 \\ a_3 \\ a_2 \\ 0 \\ 0} - a_2 b_1^2 \,\myvector{T_2^2T_3^2}{0 \\ 0 \\ 0 \\ 0 \\ 0 \\ 0 \\ 0 \\ a_3 \\ a_2 \\ 0}
\\
+ a_3^2 b_1 \, \myvector{T_1T_2^3}{0 \\ 0 \\ 0 \\ a_1c \\ 0 \\ 0 \\ b_1 \\ 0 \\ 0 \\ 0} - a_2 a_3 b_1 \, \myvector{T_1 T_2^2 T_3}{0 \\ 0 \\ 0 \\ a_2d \\ a_1c \\ 0 \\ 0 \\ b_1 \\ 0 \\ 0} + a_2^2 b_1 \,\myvector{T_2^2T_3^2}{0 \\ 0 \\ 0 \\ 0 \\ a_2d \\ a_1c \\ 0 \\ 0 \\ b_1 \\ 0}
\end{align*}
	Similarly,
$$f_3^2 = \left( a_1 b_2 - a_2 b_1 \right)^2 = a_1^2b_2^2 - 2 a_1 a_2 b_1 b_2 + a_2^2 b_1^2 = a_1^4 c^2 - 2 a_1^2 a_2 b_1 c + a_2^2 b_1^2,$$~and

\begin{align*}
\begin{bmatrix} 0 \\ 0 \\ 0 \\ 0 \\ 0 \\ 0 \\ 0 \\ 0 \\ 0 \\ 12 f_3^2 \end{bmatrix} = 
18 {a}_{2} {a}_{3}^{2} c d \,\myvector{T_1^3 T_2}{a_2 \\ a_1 \\ 0 \\ 0 \\ 0 \\ 0 \\ 0 \\ 0 \\ 0 \\ 0} + 9 {a}_{2}^{2} {a}_{3} c d \,\myvector{T_1^3 T_3}{a_3 \\ 0 \\ a_1 \\ 0 \\ 0 \\ 0 \\ 0 \\ 0 \\ 0 \\ 0} + \begin{pmatrix} - 8 {a}_{1} {a}_{3}^{2} c d \\ 6 {a}_{2}^{2} {a}_{3} d^{2} \end{pmatrix} \,\myvector{T_1^2 T_2^2}{0 \\ a_2 \\ 0 \\ a_1 \\ 0 \\ 0 \\ 0 \\ 0 \\ 0 \\ 0} 
+ \begin{pmatrix} 8 {a}_{1} {a}_{2} {a}_{3} c d \\ + 3 {a}_{2}^{3} d^{2} \end{pmatrix} \,\myvector{T_1^2 T_2 T_3}{0 \\ a_3 \\ a_2 \\ 0 \\ a_1 \\ 0 \\ 0 \\ 0 \\ 0 \\ 0}
\\
 + \begin{pmatrix} -8 {a}_{1} {a}_{2}^{2} c d \\ + 27 {a}_{2} {a}_{3} {b}_{1} c \end{pmatrix} \,\myvector{T_1^2 T_3^2}{0 \\ 0 \\ a_3 \\ 0 \\ 0 \\ a_1 \\ 0 \\ 0 \\ 0 \\ 0} + 13 {a}_{1} {a}_{2} {a}_{3} d^{2} \,\myvector{T_1 T_2^3}{0 \\ 0 \\ 0 \\ a_2 \\ 0 \\ 0 \\ a_1 \\ 0 \\ 0 \\ 0} + \begin{pmatrix} -11 {a}_{1} {a}_{2}^{2} d^{2} \\ +18 {a}_{2} {a}_{3} {b}_{1} d \end{pmatrix} \,\myvector{T_1 T_2^2 T_3}{0 \\ 0 \\ 0 \\ a_3 \\ a_2 \\ 0 \\ 0 \\ a_1 \\ 0 \\ 0} + 18 {a}_{2}^{2} {b}_{1} d \,\myvector{T_1 T_2 T_3^2}{0 \\ 0 \\ 0 \\ 0 \\ a_3 \\ a_2 \\ 0 \\ 0 \\ a_1 \\ 0}
\\
-36 {a}_{1}^{2} {a}_{3} d^{2} \,\myvector{T_2^4}{0 \\ 0 \\ 0 \\ 0 \\ 0 \\ 0 \\ a_2 \\ 0 \\ 0 \\ 0} + \begin{pmatrix} 11 {a}_{1}^{2} {a}_{2} d^{2} \\ -8 {a}_{1} {a}_{3} {b}_{1} d \end{pmatrix} \,\myvector{T_2^3 T_3}{0 \\ 0 \\ 0 \\ 0 \\ 0 \\ 0 \\ a_3 \\ a_2 \\ 0 \\ 0}
+ \begin{pmatrix} -12 {a}_{1}^{3} c d \\ -2 {a}_{1} {a}_{2} {b}_{1} d
 \end{pmatrix} \,\myvector{T_2^2 T_3^2}{0 \\ 0 \\ 0 \\ 0 \\ 0 \\ 0 \\ 0 \\ a_3 \\ a_2 \\ 0} + 27 {a}_{2} {b}_{1}^{2} \,\myvector{T_2^2 T_3^2}{0 \\ 0 \\ 0 \\ 0 \\ 0 \\ 0 \\ 0 \\ 0 \\ a_3 \\ a_2} -27 {a}_{2} {a}_{3}^{2} c  \,\myvector{T_1^3 T_3}{a_2d \\ 0 \\ b_1 \\ 0 \\ 0 \\ 0 \\ 0 \\ 0 \\ 0 \\ 0}
 \\
 -18 {a}_{2} {a}_{3}^{2} d \,\myvector{T_1^2 T_2^2}{0 \\ a_1c \\ 0 \\ b_1 \\ 0 \\ 0 \\ 0 \\ 0 \\ 0 \\ 0}
-9 {a}_{2}^{2} {a}_{3} d \,\myvector{T_1^2 T_2 T_3}{0 \\ a_2d \\ a_1c \\ 0 \\ b_1 \\ 0 \\ 0 \\ 0 \\ 0 \\ 0} -3 {a}_{2}^{3} d \,\myvector{T_1^2 T_3^2}{0 \\ 0 \\ a_2d \\ 0 \\ 0 \\ b_1 \\ 0 \\ 0 \\ 0 \\ 0} + 8 {a}_{1} {a}_{3}^{2} d \,\myvector{T_1 T_2^3}{0 \\ 0 \\ 0 \\ a_1c \\ 0 \\ 0 \\ b_1 \\ 0 \\ 0 \\ 0} -8 {a}_{1} {a}_{2} {a}_{3} d \,\myvector{T_1 T_2^2 T_3}{0 \\ 0 \\ 0 \\ a_2d \\ a_1c \\ 0 \\ 0 \\ b_1 \\ 0 \\ 0}
 \end{align*}
\[
+\begin{pmatrix} 8 {a}_{1} {a}_{2}^{2} d \\ -27 {a}_{2} {a}_{3} {b}_{1}
 \end{pmatrix} \,\myvector{T_1 T_2 T_3^2}{0 \\ 0 \\ 0 \\ 0 \\ a_2d \\ a_1c \\ 0 \\ 0 \\ b_1 \\ 0} -15 {a}_{2}^{2} {b}_{1} \,\myvector{T_1 T_3^3}{0 \\ 0 \\ 0 \\ 0 \\ 0 \\ a_2d \\ 0 \\ 0 \\ 0 \\ b_1} + 12 {a}_{1}^{2} {a}_{3} d \,\myvector{T_2^3 T_3}{0 \\ 0 \\ 0 \\ 0 \\ 0 \\ 0 \\ a_2d \\ a_1c \\ 0 \\ 0} + \begin{pmatrix} 12 {a}_{1}^{3} c \\ -24 {a}_{1} {a}_{2} {b}_{1}
 \end{pmatrix} \,\myvector{T_2 T_3^3}{0 \\ 0 \\ 0 \\ 0 \\ 0 \\ 0 \\ 0 \\ 0 \\ a_2d \\ a_1c}.
 \]
 Finally,
 $$f_1^2 = \left(a_2b_3 - a_3b_2\right)^2 = \left( a_2^2d - a_1a_3c \right)^2 = a_2^4 d^2 - 2 a_1 a_2^2 a_3 cd + a_1^2 a_3^2 c^2,$$
	and
	$$\begin{bmatrix} f_1^2 \\ 0 \\ 0 \\ 0 \\ 0 \\ 0 \\ 0 \\ 0 \\ 0 \\ 0 \end{bmatrix} =
	\begin{pmatrix} a_1a_3^2c^2 \\ -2a_2^2a_3cd \end{pmatrix}
	\myvector{T_1^4}{a_1 \\ 0 \\ 0 \\ 0 \\ 0 \\ 0 \\ 0 \\ 0 \\ 0 \\ 0}
	+ a_2^3d^2 \, \myvector{T_1^3 T_2}{a_2 \\ a_1 \\ 0 \\ 0 \\ 0 \\ 0 \\ 0 \\ 0 \\ 0 \\ 0 }
	- a_1a_2^2d^2 \, \myvector{T_1^2T_2^2}{0 \\ a_2 \\ 0 \\ a_1 \\ 0 \\ 0 \\ 0 \\ 0 \\ 0 \\ 0} 
	+ a_1^2a_2d^2 \, \myvector{T_1 T_2^3}{0 \\ 0 \\ 0 \\ a_2 \\ 0 \\ 0 \\ a_1 \\ 0 \\ 0 \\ 0}
	- a_1^3d^2 \, \myvector{T_2^4}{0 \\ 0 \\ 0 \\ 0 \\ 0 \\ 0 \\ a_2 \\ 0 \\ 0 \\ 0}.$$

Finally, to see that $P(a,b,c)$ satisfies these conditions for all $a, b, c$, note that
	$$P(a,b,c) = I_2 \begin{pmatrix} x^{\alpha_3} & y^{\beta_1} & z^{\gamma_2} \\ z^{\gamma_1} & x^{\alpha_2} & y^{\beta_3} \end{pmatrix},$$
	and that up to permuting rows or columns, this matrix satisfies the required conditions. Indeed, note that up to switching the rows, two of the top row entries divide corresponding entries in the bottom row. By possibly switching the columns, these two entries can be made to be $a_1$ and $a_2$. Then either $a_1 | b_2$ or $a_2 | b_1$; if it is the second option, then switch the first two columns to get $a_1 | b_2$.
\end{proof}

The downside of this method is that we can only study each containment $I^{(a)} \subseteq I^b$ separately. On the other hand, if a containment such as $I^{(4)} \subseteq I^3$ holds, Theorem \ref{thm stable} can then be applied to yield containments for large $b$.
In fact, there are classes of space monomial curves that satisfy $P^{(4)} \subseteq P^3$. By Theorem \ref{thm stable}, such ideals satisfy $P^{(2n-2)} \subseteq P^n$ for all $n \geqslant 6$.

\begin{theorem}\label{thm sufficient 34}
	Suppose that $\ch k \neq 2$. If $a_1 | b_2 | a_1^2$, $a_2|b_3$, and either $a_3 | b_1$ or $b_1 | a_3$, then $I^{(4)} \subseteq I^3$.
	In particular, if
	$$P = P(a,b,c) = I_2 \begin{pmatrix} x^{\alpha_3} & y^{\beta_1} & z^{\gamma_2} \\ z^{\gamma_1} & x^{\alpha_2} & y^{\beta_3} \end{pmatrix}$$	
	is such that 
	$$\alpha_3 \leqslant \alpha_2 \leqslant 2 \alpha_3 \textrm{ and } \beta_1 \leqslant \beta_3,$$
	then $P^{(4)} \subseteq P^3$.
\end{theorem}

\begin{proof}
Given Remark \ref{remark char 3}, we do not need to exclude characteristic $3$, since $\frac{4!}{3!}$ is not divisible by $3$. Since $a_1 | b_2$, $a_2 | b_3$ and $b_2 | a_1^2$,
\begin{align*}
\begin{bmatrix} f_1 \\ 0 \\ 0 \\ 0 \\ 0 \\ 0 \end{bmatrix} = 
-a_3 \, \frac{b_2}{a_1} \, \myvector{T_1^3}{a_1 \\ 0 \\ 0 \\ 0 \\ 0 \\ 0}
+b_3 \, \myvector{T_1^2T_2}{a_2 \\ a_1 \\ 0 \\ 0 \\ 0 \\ 0} - a_1 \frac{b_3}{a_2} \,\myvector{T_1T_2^2}{0 \\ a_2 \\ 0 \\ a_1 \\ 0 \\ 0} + \frac{a_1^2}{b_2} \frac{b_3}{a_2} \, \myvector{\, T_2^3 \,}{0 \\ 0 \\ 0 \\ b_2 \\ 0 \\ 0} .	
\end{align*}
Similarly, since $a_1 | b_2$ and $a_2 | b_3$,
$$\begin{bmatrix} 0 \\ 0 \\ 0 \\ 2f_2 \\ 0 \\ 0 \end{bmatrix} = 
-a_3 \, \frac{b_2}{a_1} \, \myvector{T_1^2T_2}{a_2 \\ a_1 \\ 0 \\ 0 \\ 0 \\ 0} 
+ a_2 \, \frac{b_2}{a_1} \,\myvector{T_1^2 T_3}{a_3 \\ 0 \\ a_1 \\ 0 \\ 0 \\ 0} 
+ b_3 \,\myvector{\, T_1T_2^2 \,}{0 \\ a_2 \\ 0 \\ a_1 \\ 0 \\ 0} 
- 3 a_1 \frac{b_3}{a_2} \, \myvector{\, T_2^2 \,}{0 \\ 0 \\ 0 \\ a_2 \\ 0 \\ 0}
+ b_1 \, \myvector{\, T_2^2 T_3 \,}{0 \\ 0 \\ 0 \\ a_3 \\ a_2 \\ 0}
+ a_3 \, \myvector{T_1T_2^2}{0 \\ b_2 \\ 0 \\ b_1 \\ 0 \\ 0}
- a_2 \, \myvector{T_1T_2T_3}{0 \\ b_3 \\ b_2 \\ 0 \\ b_1 \\ 0}.$$
Finally, if $b_1 | a_3$, then the following identity holds:
\begin{align*}
\begin{bmatrix} 0 \\ 0 \\ 0 \\ 0 \\ 0 \\ f_3 \end{bmatrix} = 
b_3 \, \frac{b_2}{a_1} \frac{a_3}{b_1} \,\myvector{T_1^3}{a_1 \\ 0 \\ 0 \\ 0 \\ 0 \\ 0}
-2 a_3 \, \frac{b_2}{a_1} \,\myvector{T_1^2T_3}{a_3 \\ 0 \\ a_1 \\ 0 \\ 0 \\ 0}
+ b_2 \,\myvector{T_1T_3^2}{0 \\ 0 \\ a_3 \\ 0 \\ 0 \\ a_1}
- b_1 \,\myvector{T_2T_3^2}{0 \\ 0 \\ 0 \\ 0 \\ a_3 \\ a_2}
\\
+ 2 \, \frac{b_2}{a_1} \, \frac{a_3^2}{b_1} \,\myvector{T_1^3}{b_1 \\ 0 \\ 0 \\ 0 \\ 0 \\ 0}
- b_3 \, \frac{a_3}{b_1} \,\myvector{T_1^2T_2}{b_2 \\ b_1 \\ 0 \\ 0 \\ 0 \\ 0}
+ a_3 \,\myvector{T_1T_2T_3}{0 \\ b_3 \\ b_2 \\ 0 \\ b_1 \\ 0} 
;
\end{align*}
and if $a_3 | b_1$, we also have
\begin{align*}
	\begin{bmatrix} 0 \\ 0 \\ 0 \\ 0 \\ 0 \\ 2f_3 \end{bmatrix} = 
a_3 \, \frac{b_3}{a_2} \,\myvector{T_1 T_2^2}{ 0 \\ a_2 \\ 0 \\ a_1 \\ 0 \\ 0}
- b_3 \,\myvector{T_1 T_2 T_3}{ 0 \\ a_3 \\ a_2 \\ 0 \\ a_1 \\ 0} 
- a_1 \, \frac{b_3}{a_2} \,\myvector{T_2^2 T_3}{ 0 \\ 0 \\ 0 \\ a_3 \\ a_2 \\ 0} 
- 3 a_2 \, \frac{b_1}{a_3} \,\myvector{T_3^3}{ 0 \\ 0 \\ 0 \\ 0 \\ 0 \\ a_3} 
+ a_2 \,\myvector{T_1^2 T_3}{ 0 \\ 0 \\ b_3 \\ 0 \\ 0 \\ b_1}   
+ 2 a_1 \,\myvector{T_2 T_3^2}{ 0 \\ 0 \\ 0 \\ 0 \\ b_3 \\ b_2}.
\end{align*}\qedhere
\end{proof}

\begin{theorem}\label{theorem 43 for a=3 and 4}
	Let $\ch k \neq 2$, $P = P(a,b,c)$. If $a=3$ or $4$ and $a<b<c$, then~$P^{(4)} \subseteq P^3$.
\end{theorem}

\begin{proof}
	If $P$ is a complete intersection, then $P^{(4)} = P^4 \subseteq P^3$, so we can assume that $P$ is not a complete intersection. In the proof of \cite[Theorem 3.14]{HunekeHilbertSymb}, Huneke shows that if $a=4$ and $P$ is not a complete intersection, then $P$ is minimally generated by the maximal minors of
	$$\begin{pmatrix} y & z & x^p \\ x^q & y^2 & z \end{pmatrix},$$
	where $b = p + 2q$ and $c = 2q+3p$. By Theorem \ref{thm sufficient 34}, we have $P^{(4)} \subseteq P^3$.
	
	A technique similar to the one used in the proof of \cite[Theorem 3.14]{HunekeHilbertSymb} can be used to determine the form of $P$ for different values of $a$. First, we note that $(P,x)$ must be of the form $(P,x) = (y^{\beta_1}z^{\gamma_1}, y^{\beta_2}, z^{\gamma_3})$, and that $R/(P,x)$ has multiplicity $a$. Following the classification in \cite{Poonen}, we can then determine all possibilities for $(P,x)$ when $a=3$, and conclude that the only possibility is $(P,x) = (x^2, y^2, xy)$. In particular, $P = I_2(M)$, where
		$$M = \begin{pmatrix} x^{\alpha_3} & y & z \\ z & x^{\alpha_2} & y \end{pmatrix}.$$
		Now any matrix of this form satisfies the conditions in Theorem \ref{thm sufficient 34}, so $P^{(4)} \subseteq P^3$.
\end{proof}

The containment $I^{(4)} \subseteq I^3$ can hold even if the symbolic Rees algebra of $I$ is not noetherian.

\begin{example}\label{non-noetherian example}
	Let $k = \mathbb{C}$ and fix an integer $n \geqslant 4$ not divisible by $3$. By \cite{NonnoetherianSymb}, 
	$$P(7n-3, (5n-2)n, 8n-3) = I_2 \begin{pmatrix} y & x^n & z^{2n-1} \\ z^n & y^2 & x^{2n-1} \end{pmatrix},$$
	Also by \cite[Corollary 1.2]{NonnoetherianSymb}, the symbolic Rees algebra of $P$ is not noetherian. By Theorem \ref{thm sufficient 34}, $P^{(4)} \subseteq P^3$, and by Theorem \ref{thm stable}, $P^{(2n-2)} \subseteq P^n$ for all $n \geqslant 6$.
\end{example}

But $P^{(4)} \subseteq P^3$ does not hold for all space monomial curves. 

\begin{example}
	Consider
	$$P(t^9,t^{11},t^{14}) = I_2 \begin{pmatrix} z & y^3 & x^{3} \\ x & z^2 & y^{2} \end{pmatrix} \subseteq \mathbb{C}[x,y,z].$$
	Macaulay2 \cite{M2} computations show that $P^{(4)} \nsubseteq P^3$, and that in fact, this is the smallest such example, meaning that for any $a \leqslant 9$, $b \leqslant 11$, and $c \leqslant 14$ such that $(a,b,c) \neq (9,11,14)$, $P(a,b,c)^{(4)} \subseteq P(a,b,c)^3$. Note that no row or column operations can result in a matrix satisfying the conditions in Theorem \ref{thm sufficient 34}.
\end{example}

\section{Acknowledgements}

This work started as part of my PhD thesis, and could not have happened without the guidance and support of my advisor, Craig Huneke, who was also the one who first pointed \cite{Seceleanu} out to me. I am very grateful to Alexandra Seceleanu for allowing me to include here her Example \ref{Alexandra's Fermat example}, for her own work that inspired this paper, and for the many conversations we have had about this project. I also want to thank Alessandro De Stefani, Brian Harbourne, Jack Jeffries, and Karen Smith for helpful discussions. Thank you also to Jack Jeffries, Vivek Mukundan and Alexandra Seceleanu for their comments and suggestions on a preliminary draft of this paper, and to Elena Guardo for kindly pointing out a typo in the published version of the paper (I was so excited about reflection arrangements I upgraded them to reflection arrangements reflection arrangements). Finally, I am grateful to the anonymous referee's extremely detailed and helpful comments, especially for improving Remark \ref{remark any C} with the observation that the invariants $\rho''$ and $\rho_a'$ coincide.

Originally, the proofs of Theorem \ref{sufficient condition for 35} and Theorem \ref{thm sufficient 34} relied on excruciating computations I did by hand. It turns out the proofs could have been found with a computer! The proofs here presented are in fact improvements found with the assistance Macaulay2 \cite{M2}, which was also used to compute many examples during this project. This Macaulay2 work became a lot more efficient thanks to discussions with Craig Huneke and Vivek Mukundan as part of a parallel joint project.


\bibliographystyle{alpha}
\bibliography{References}
\end{document}